\documentclass{amsart}
\usepackage[top=40truemm,bottom=35truemm,left=42truemm,right=42truemm]{geometry}
\usepackage{amsmath,amssymb}	% to deal with mathematics

\usepackage{graphicx}		  % to deal with graphicals

\usepackage{amscd} % Morimoto added

\theoremstyle{plain}
\newtheorem{thm}{Theorem}
\newtheorem{lem}{Lemma}
\newtheorem{prop}{Proposition}
\newtheorem{cor}{Corollary}

\theoremstyle{definition}
\newtheorem{defn}{Definition}
\newtheorem{rem}{Remark}

\newtheorem{example}{Example} % author added this command

\begin{document}

\title[On weakly reflective PF  submanifolds in Hilbert spaces]{On weakly reflective PF submanifolds \\ in Hilbert spaces}

\author[M.\ Morimoto]{Masahiro Morimoto}

\address[M. Morimoto]{Department of Mathematics, Graduate School of Science, Osaka City University, 3-3-138 Sugimoto, Sumyoshi-ku, Osaka 558-8585 JAPAN.}

\email{mmasahiro0408@gmail.com}



\thanks{The author was partly supported by the Grant-in-Aid for JSPS Research Fellow (No.18J14857).}

%%%%%%    TEXT START    %%%%%%

\maketitle

\begin{abstract}
A weakly reflective submanifold is a minimal submanifold of a  Riemannian manifold which has a certain symmetry at each point. In this paper we introduce this notion into a class of proper Fredholm (PF)  submanifolds in Hilbert spaces and  show that there exist  so many infinite dimensional weakly reflective PF submanifolds in Hilbert spaces. In particular each fiber of the parallel transport map is shown to be weakly reflective. These imply that in infinite dimensional Hilbert spaces there exist so many homogeneous minimal submanifolds which are not totally geodesic, unlike  in the finite dimensional Euclidean case. 
\end{abstract}

\section*{Introduction}

In \cite{IST09} Ikawa, Sakai and Tasaki introduced a concept of weakly reflective submanifolds, which constitute a  special class of minimal submanifolds in finite dimensional Riemannian manifolds. This class is related to other classes of minimal submanifolds as follows:
\begin{equation*}
\begin{array}{ccccccc}
&&\text{totally geodesic}&&&& \vspace{-3mm}
\\
&\rotatebox{30}{$\hspace{-1.8mm}\Rightarrow$}&&\rotatebox{-30}{$\hspace{-1.8mm}\Rightarrow$}&&&\vspace{-2mm}
\\
\text{reflective}&&&&\text{austere}&\Rightarrow&\text{minimal} \vspace{-2mm}
\\
&\rotatebox{-30}{$\hspace{-1.8mm}\Rightarrow$}&&\rotatebox{30}{$\hspace{-1.8mm}\Rightarrow$}&&& \vspace{-3mm}
\\
&&\text{\emph{weakly reflective}}&&&&
\end{array}
\end{equation*}
Let $\tilde{M}$ be a finite dimensional Riemannian manifold. A \emph{reflective} submanifold of $\tilde{M}$ is defined as a connected component of the fixed point set of an involutive isometry of $\tilde{M}$.  An immersed submanifold $M$ of $\tilde{M}$ is called \emph{weakly reflective}  (\cite{IST09}) if for each $p \in M$ and each $\xi  \in T^\perp_p M$, there exists an isometry $\nu_\xi$ of $\tilde{M}$ which satisfies 
\begin{equation*}
\nu_\xi(p) = p,
\quad
(d \nu_\xi)_p \xi = -\xi,
\quad
\nu_\xi(M) = M.
\end{equation*}
Here we call such an isometry $\nu_\xi$   a \emph{reflection}  of $M$ at $p$ with respect to $\xi$. An immersed submanifold $M$ of $\tilde{M}$ is called \emph{austere} (\cite{HL82})  if for each $\xi \in T^\perp M$ the set of eigenvalues with their multiplicities of the shape operator $A_{\xi}$ is invariant under the multiplication by $(-1)$. 

It is an interesting problem to study submanifold geometry of orbits under  isometric actions of Lie groups and to determine their weakly reflective orbits. Podest\`a (\cite{Pod97}) essentially proved that any singular orbit of a cohomogeneity one action is weakly reflective. Ikawa, Sakai and Tasaki (\cite{IST09}) classified weakly reflective orbits and austere orbits of $s$-representations.  Ohno (\cite{Ohno16}) gave sufficient conditions for orbits of Hermann actions to be weakly reflective.  Recently Enoyoshi (\cite{Eno18}) showed that there exists a unique weakly reflective orbit among the principal orbits of the cohomogeneity one action of the exceptional Lie group $G_2$ on a Grassmann manifold $\widetilde {\operatorname{Gr}}_3 (\operatorname{Im}\mathbb{O})$.
Notice that at present all known examples of weakly reflective submanifolds are homogeneous, that is, orbits of isometric actions by certain Lie groups. 

The purpose of this paper is to introduce the concept of weakly reflective submanifolds into a class of     \emph{proper Fredholm} (PF) submanifolds in Hilbert spaces and   show that many infinite dimensional weakly reflective PF submanifolds are obtained from finite dimensional weakly reflective submanifolds in compact normal homogeneous spaces through the \emph{parallel transport map}.

The study of submanifolds in Hilbert spaces was initiated by Terng (\cite{Ter89}). In order to apply infinite dimensional Morse theory to submanifolds in Hilbert spaces, she introduced a class of \emph{proper Fredholm} (PF) submanifolds (cf.\  Section 1) and studied isoparametric PF submanifolds.  In particular she gave examples of PF submanifolds which are orbits of the gauge transformations.  Such examples were extended by Pinkall and Thorbergsson (\cite{PiTh90}) and eventually  reformulated by Terng as \emph{$P(G,H)$-actions} (\cite{Ter89}). More generally, PF submanifolds can be obtained through  the \emph{parallel transport map} $\Phi_K : V_\mathfrak{g} \rightarrow G/K$ (\cite{KT93}, \cite{TT95}),
which is a Riemannian submersion of 
a Hilbert space $V_\mathfrak{g} := H^0([0,1], \mathfrak{g})$ onto a compact normal homogeneous space $G/K$ (see (\ref{ptm2}) in  Section 2). It is known that if $N$ is a closed submanifold of $G/K$, then the inverse image $\Phi_K^{-1}(N)$ is a PF submanifold of $V_\mathfrak{g}$. Nowadays the parallel transport map is known as a precious tool for  obtaining PF submanifolds.

In this paper we first define weakly reflective PF submanifolds similarly to the finite dimensional case. Then under suitable assumptions we show that if $N$ is a weakly reflective submanifold of $G/K$, then the inverse image $\Phi_K^{-1}(N)$ is a weakly reflective PF submanifold of $V_\mathfrak{g}$ (Theorems \ref{type1}, \ref{type2}, \ref{type2.2}, \ref{type3}). From these results and examples of weakly reflective submanifolds in $G/K$, 
we obtain many examples of homogeneous weakly reflective PF submanifolds (Examples \ref{examp0},  \ref{examp1}, \ref{examp2}, \ref{Ohno2},  \ref{examp3}, \ref{examp4}).
Moreover we see that these weakly reflective PF submanifolds are \emph{not} totally geodesic at all except for rare cases clarified in Theorem \ref{tgeod8}. As a consequence those  show (Remark \ref{homgmin}) that in infinite dimensional Hilbert spaces there exist so many homogeneous minimal submanifolds which are \emph{not} totally geodesic, unlike in the finite dimensional Euclidean case (\cite{Sca02}). 

This paper is organized as follows. In Section 1 we introduce weakly reflective PF submanifolds and related notions. In Section \ref{prel} we prepare the setting of $P(G,H)$-actions and the parallel transport map $\Phi_K$.  In Section 3 we study the second fundamental form and the shape operator of a PF submanifold obtained  through $\Phi_K$. In Section 4 we give some criteria for so obtained PF submanifolds to be totally geodesic. In Section 5 we define the canonical reflection of $V_\mathfrak{g}$ and prove that each fiber of $\Phi_K$ is a weakly reflective PF submanifold of $V_\mathfrak{g}$.  In Section 6 under suitable assumptions we show that 
a submanifold $N$ of a compact normal homogeneous space $G/K$
is weakly reflective if and only if the inverse image $\Phi^{-1}_K(N)$ is a weakly reflective PF submanifold of $V_\mathfrak{g}$.
In Section 7 supposing that $G/K$ is a Riemannian symmetric space of compact type we show that for \emph{any} weakly reflective submanifold $N$ of $G/K$ the inverse image $\Phi^{-1}_K(N)$ is a weakly reflective PF submanifold of $V_\mathfrak{g}$.
%%%%%%%%%%%%%%%%%%%%%%%%

\section{Weakly reflective PF submanifolds and their minimality}

Let $V$ be a separable Hilbert space over $\mathbb{R}$. An immersed submanifold $M$ of finite codimension in $V$ is called  \emph{proper Fredholm} (PF) (\cite{Ter89}) if the restriction of the end point map $T^\perp M \rightarrow V$, 
$(p,\xi) \mapsto p + \xi$ 
to a normal disk bundle of any finite radius is proper and Fredholm. As in the following, weakly reflective PF submanifolds and related notions are defined similarly to  the finite dimensional case, except for minimal submanifolds. Note that (\cite[p.\ 16]{Ter89}) for a PF submanifold $M$, its shape operator in the direction of each normal vector 
 is a self-adjoint compact operator on  a Hilbert space, which is not of trace class in general. 

\begin{defn}
Let $M$ be a PF submanifold of $V$. 
$M$ is called \emph{reflective} if it is a connected component of the fixed point set of an involutive isometry of $V$. $M$ is called \emph{totally geodesic} if its second fundamental form is identically zero. $M$ is called \emph{weakly reflective}  if for each $p \in M$ and each $\xi  \in T^\perp_p M$, there exists an isometry $\nu_{\xi}$ of $V$ which satisfies 
\begin{equation*}
\nu_\xi(p) = p,
\quad
(d \nu_\xi)_p \xi = -\xi,
\quad
\nu_\xi(M) = M.
\end{equation*}
Here we call such an isometry $\nu_\xi$ a \emph{reflection}  of $M$ at $p \in M$ with respect to $\xi$.  $M$ is  called \emph{austere}  if for each $\xi \in T^\perp M$ the set of eigenvalues with their multiplicities  of the shape operator $A_{\xi}$ is invariant under the multiplication by $(-1)$. 

At present, three kinds of definitions of  `minimal' PF submanifolds are known (King-Terng \cite{KT93}, Heintze-Liu-Olmos \cite{HLO06}, Koike \cite{Koi02}). 

Let $A_\xi$ be the shape operator of $M$ in the direction of $\xi \in T^\perp M$. We denote by $\mu_1 \leq \mu_2 \leq \cdots <0< \cdots \leq  \lambda_2 \leq \lambda_1$ its non-zero eigenvalues repeated with multiplicities. $A_\xi$ is called \emph{$\zeta$-regularizable} (\cite{KT93}) if $\sum_{k} \lambda_k^s + \sum_{k} |\mu_k|^s < \infty$ for all $s >1$ and $\operatorname{tr}_{\zeta} A_\xi := \lim_{s \searrow1}(\sum_{k} \lambda_k^s - \sum_{k} |\mu_k|^s)$ exists. Then we call $\operatorname{tr}_\zeta A_\xi$ the \emph{$\zeta$-regularized mean curvature}  in the direction of $\xi$.  $M$ is called \emph{$\zeta$-regularizable} if $A_\xi$ is $\zeta$-regularizable for all $\xi \in T^\perp M$. If $M$ is $\zeta$-regularizable and $\operatorname{tr}_{\zeta} A_\xi$ vanishes for all $\xi \in T^\perp M$,  we say that $M$ is \emph{$\zeta$-minimal}.

$A_\xi$ is called  \emph{regularizable} (\cite{HLO06}) if $\operatorname{tr} A_\xi^2 < \infty$ and $\operatorname{tr}_{r} A_\xi := \sum_{k=1}^\infty (\lambda_k + \mu_k)$ converges, where we regard $\lambda_k$ or $\mu_k$ as zero if there are less than $k$ positive or negative eigenvalues, respectively. Then we call 
$
\operatorname{tr}_{r} A_\xi
$
the \emph{regularized mean curvature} in the direction of $\xi$.  $M$ is called  \emph{regularizable} if $A_\xi$ is  regularizable for all $\xi \in T^\perp M$. If $M$ is regularizable and $\operatorname{tr}_{r} A_\xi$ vanishes for all $\xi \in T^\perp M$,  we say that $M$ is \emph{r-minimal}.

$M$ is called \emph{formally minimal} (\cite{Koi02}) (shortly, \emph{f-minimal})  if $\operatorname{tr}_f A_\xi:= \sum_{k =1}^\infty m_k \kappa_k$ converges to zero for each unit normal vector $\xi \in T^\perp M$, where $\{\kappa_k\}_{k=1}^\infty$ denotes the set of all distinct non-zero eigenvalues of $A_\xi$ arranged so that $|\kappa_k|> |\kappa_{k+1}|$ or $\kappa_k = - \kappa_{k+1} \geq 0$, and $m_k$ is the multiplicity of $\kappa_k$.

\end{defn}

Note that each isometry of $V$ is written by $x \mapsto P x + q$, where $P$ is an orthogonal transformation of $V$  and $q \in V$. Note also that if $M$ is connected, then the following  are equivalent:  (i) $M$ is reflective, (ii) $M$ is totally geodesic, (iii) $M$ is an affine subspace of $V$. Moreover we have the following relation for PF submanifolds.
\begin{equation*}
\begin{array}{ccccccc}
&&\text{totally geodesic}&&&&\text{$\zeta$-minimal} \vspace{-3mm}
\\
&\rotatebox{30}{$\hspace{-1.8mm}\Leftrightarrow$}&&\rotatebox{-30}{$\hspace{-1.8mm}\Rightarrow$}&&\rotatebox{25}{- ? -}&\vspace{-2mm}
\\
\text{reflective}&&&&\text{austere}&\text{- ? -}& \text{r-minimal}\vspace{-2mm}
\\
&\rotatebox{-30}{$\hspace{-1.8mm}\Rightarrow$}&&\rotatebox{30}{$\hspace{-1.8mm}\Rightarrow$}&&\rotatebox{-25}{- ? -}& \vspace{-3mm}
\\
&&\text{weakly reflective}&&&&\text{f-minimal}
\end{array}
\end{equation*}
We do not know whether an austere PF submanifold is $\zeta$-minimal, r-minimal or f-minimal in the infinite dimensional case. It is  clear that regularizable austere PF submanifolds are $r$-minimal.
It also follows easily from the definition that $\zeta$-regularizable austere PF submanifolds are both $\zeta$-minimal and r-minimal.

%%%%%%%%%%%%%%%%%%%%%%%%

\section{$P(G,H)$-actions and the parallel transport map
} \label{prel}

In this section we prepare the setting of $P(G,H)$-actions and the parallel transport map.

Let $\mathfrak{G}$ be a  Hilbert Lie group, $\mathfrak{M}$ a Hilbert manifold. A $\mathfrak{G}$-action on $\mathfrak{M}$ is called \emph{proper Fredholm} (PF) (\cite{PT88}) if a map $\mathfrak{G} \times \mathfrak{M} \rightarrow \mathfrak{M} \times \mathfrak{M}$, $(g, p) \mapsto (g \cdot p, \, p)$ is proper, and for each $p \in \mathfrak{M}$ a  map $\mathfrak{G} \rightarrow \mathfrak{M}$, $g \mapsto g \cdot p$ is Fredholm. If an infinite dimensional Hilbert Lie group action on a separable Hilbert space $V$ is isometric and PF, then each of its orbits is a PF submanifold of $V$ (\cite[Theorem 7.1.6]{PT88}).  

Let $G$ be a connected compact Lie group with Lie algebra $\mathfrak{g}$. Fix an $\operatorname{Ad}(G)$-invariant inner product of $\mathfrak{g}$ and equip  the corresponding  bi-invariant Riemannian metric with $G$. For simplicity of notation, we regard $G$ as a subgroup of a general linear group.   

We denote by $V_\mathfrak{g} := H^0([0,1], \mathfrak{g})$ a Hilbert space of all Sobolev $H^0$-paths (i.e.\ $L^2$-paths)  in $\mathfrak{g}$ parametrized by $t \in [0,1]$. Also denote by $\mathcal{G} := H^1([0,1], G)$ a Hilbert Lie group of all Sobolev $H^1$-paths in $G$ parametrized by $t \in [0,1]$. We use $\hat{\ }$ to denote  a map which corresponds to each $x \in \mathfrak{g}$ (resp.\ $a \in G$) the constant path $\hat{x} \in V_\mathfrak{g}$ (resp.\ $\hat{a} \in \mathcal{G}$).  $\mathcal{G}$ acts on $V_\mathfrak{g}$ via the left gauge transformations: 
\begin{equation*}
g * u := g u g^{-1} - g' g^{-1}, 
\quad
g \in \mathcal{G}, \ u \in V_\mathfrak{g}. 
\end{equation*}
The differential $d (g*)$ of the transformation $g* : V \rightarrow V$, $u \mapsto g*u$ is given by
$d(g*) (X) = g X g^{-1}$ for 
$X \in T_{\hat{0}} V_\mathfrak{g}\cong V_\mathfrak{g}$. We know that the $\mathcal{G}$-action on $V_\mathfrak{g}$ is \emph{isometric,  transitive and PF} (\cite[p.\ 24]{Ter89}). 

Let $H$ be a closed subgroup of $G \times G$ with Lie algebra $\mathfrak{h}$.  Define a Lie subgroup $P(G,H)$ of $\mathcal{G}$ by 
\begin{equation*}
P(G,H) := 
\{g \in \mathcal{G} \mid (g(0), g(1)) \in H\}.
\end{equation*}
with Lie algebra $\operatorname{Lie}P(G,H) := \{Z \in H^1([0,1],\mathfrak{g}) \mid (Z(0), Z(1)) \in \mathfrak{h}\}$. The induced action of $P(G,H)$ on $V_\mathfrak{g}$ is called the \emph{$P(G, H)$-action} (\cite{Ter95}). Note that $P(G, H)$ is an inverse image of $H$ under the Lie group homomorphism
\begin{equation*}
\Psi : \mathcal{G} \rightarrow G \times G,
\quad
g \mapsto (g(0), g(1)).
\end{equation*}
Since $\Psi$ is a submersion, it follows that the  
$P(G,H)$-action on $V_\mathfrak{g}$ is isometric and PF (\cite[p.\ 132]{Ter95}).  It also follows that if $H = \{e\} \times G$, then $P(G, \{e\} \times G)$ acts on $V_\mathfrak{g}$ transitively and freely (\cite[p.\ 685]{TT95}). Similarly $P(G, G \times \{e\})$ acts on $V_\mathfrak{g}$ transitively and freely (see also \eqref{canocom2} in Section 5).

The natural left action of $H$ on $G$  is defined by 
\begin{equation} \label{Haction}
(b_1, b_2) \cdot a := b_1 a b_2^{-1},
\quad
a \in G, \ (b_1, b_2) \in H.
\end{equation}
The $P(G, H)$-action is closely related to this $H$-action through the \emph{parallel transport map} (\cite{KT93}), which is defined as follows.  Let $E : V_\mathfrak{g} \rightarrow P(G, \{e\} \times G)$, $u \mapsto E_u$ be a map defined by a unique solution to the linear  ordinary differential equation
\begin{equation*}
\left\{
\begin{array}{cc}
E^{-1}_u E'_u = u,
\\
E_u(0) = e.
\end{array}
\right. 
\end{equation*}
The parallel transport map $\Phi : V_\mathfrak{g} \rightarrow G$ is defined by 
\begin{equation*}
\Phi(u) := E_u(1),
\quad
u \in V_\mathfrak{g}.
\end{equation*}
It follows (\cite[p.\ 133]{Ter95}) that for $g \in \mathcal{G}$ and $u \in V_\mathfrak{g}$, 
\begin{equation} \label{equiv4}
\text{(i)} \ \Phi(g * u) = \Psi(g) \cdot \Phi(u),
\qquad
\text{(ii)} \ P(G,H) * u  = \Phi^{-1}(H \cdot \Phi(u)).
\end{equation}
In other words, the following commutative diagram holds.
\begin{equation*}
\begin{array}{ccccccccc}
\mathcal{G} &\supset& P(G,H) & \curvearrowright & V_\mathfrak{g}& \supset & P(G,H) * u & =&\Phi^{-1}(H \cdot \Phi(u)) \vspace{2mm}
\\
\Psi \downarrow \ \ \ &&\Psi \downarrow \ \ \ &&\Phi \downarrow \ \ \ &&\Phi \downarrow \ \ \ &&\vspace{2mm}
\\
G \times G& \supset & H & \curvearrowright & G & \supset & H \cdot \Phi(u)&&
\end{array}
\end{equation*}

The differential of a map 
\begin{equation} \label{isometry}
P(G , \{e\} \times G) \rightarrow V_\mathfrak{g}, 
\quad 
g \mapsto g^{-1} * \hat{0}
\end{equation}
is given by $T_{\hat{e}} P(G, \{e\} \times G) \rightarrow T_{\hat{0}} V_\mathfrak{g}$, $Z \mapsto Z'$. We know (\cite[Proposition 3.2]{KT93}) that (\ref{isometry}) is an isometric diffeomorphism with respect to the right invariant Riemannian metric $\langle \cdot , \cdot \rangle$ on $P(G, \{e\} \times G)$ defined by 
\begin{equation*}
\langle Z, W \rangle 
:=
\langle Z', W' \rangle_{L^2},
\quad
Z, W \in T_{\hat{e}} P(G, \{e\} \times G).
\end{equation*}
Since (\ref{isometry}) is the inverse map of $E$, the differential $(d\Phi )_{\hat{0}} : T_{\hat{0}} V_\mathfrak{g} \rightarrow \mathfrak{g}$ of $\Phi$ at $\hat{0} \in V_\mathfrak{g}$ is  given by
\begin{equation*}
(d \Phi)_{\hat{0}} (X) = \int_0^1 X(t) dt,
\quad
X \in T_{\hat{0}} V_\mathfrak{g} \, \cong V_\mathfrak{g}.
\end{equation*}
Hence the following orthogonal direct sum decomposition holds.
\begin{equation} \label{decomp0}
T_{\hat{0}} V_\mathfrak{g} = \hat{\mathfrak{g}} \oplus \operatorname{Ker}(d\Phi)_{\hat{0}}, 
\quad
{\textstyle
X = \left(\int_0^1 X(t) dt \right) \oplus \left(X -\int_0^1 X(t) dt \right).
}
\end{equation}
Moreover the following facts are known (\cite[p.\ 686]{TT95}, \cite[Lemma 5.1]{TT95}).
\begin{prop}\label{prop1} \  
\begin{enumerate}
\item[\textup{(i)}] $\Phi$ is a Riemannian submersion.

\item[\textup{(ii)}] $P(G, \{e\} \times \{e\})$ acts on each fiber of $\Phi$ transitively and freely.

\item[\textup{(iii)}] $\Phi$ is a principal $P(G, \{e\} \times \{e\})$-bundle.

\item[\textup{(iv)}] Any two fibers of $\Phi$ are congruent under the isometries on $V_\mathfrak{g}$. 

\item[\textup{(v)}]  If $N$ is a closed submanifold of $G$, then $\Phi^{-1}(N)$ is a PF submanifold of $V_\mathfrak{g}$. 
\end{enumerate}
\end{prop}
\noindent Furthermore the following properties are known (\cite[Theorem 4.12]{KT93}, \cite[Lemma 5.2]{HLO06}).
\begin{prop} \label{prop2}
Let $N$ be a closed submanifold of $G$. Then 
\begin{enumerate}
\item[\textup{(i)}] $\Phi^{-1}(N)$ is both $\zeta$-regularizable and regularizable.

\item[\textup{(ii)}] For each $X \in T^\perp \Phi^{-1}(N)$ the following coincide:\\
\textup{(a)} The $\zeta$-regularized mean curvature of $\Phi^{-1}(N)$ in the direction of $X $,\\
\textup{(b)} The regularized mean curvature of $\Phi^{-1}(N)$ in the direction of $X$,\\
\textup{(c)} The mean curvature of $N$ in the direction of $d \Phi(X) \in T^\perp N$.
\item[\textup{(iii)}] The following are equivalent: \\
\textup{(a)} $\Phi^{-1}(N)$ is $\zeta$-minimal, \textup{(b)} $\Phi^{-1}(N)$ is $r$-minimal, \textup{(c)} $N$ is minimal.
\end{enumerate}
\end{prop}

Let $K$ be a closed subgroup of $G$ with Lie algebra $\mathfrak{k}$. Denote by $\mathfrak{g}= \mathfrak{k} + \mathfrak{m}$ the orthogonal direct sum  decomposition. Restricting the $\operatorname{Ad}(G)$-invariant inner product of $\mathfrak{g}$ to $\mathfrak{m}$ we  define the induced $G$-invariant Riemannian metric on a homogeneous space $G/K$. Thus $G/K$ is a compact normal homogeneous space. We denote by $\pi : G \rightarrow G/K$ the natural projection, which is a 
Riemannian submersion with totally geodesic fiber. For each $x \in \mathfrak{g}$, $x_{\mathfrak{k}}$ and $x_{\mathfrak{m}}$ denote the 
$\mathfrak{k}$- and $\mathfrak{m}$-components, respectively.

The \emph{parallel transport map $\Phi_K$ over $G/K$} is defined by 
\begin{equation} \label{ptm2} 
\Phi_K := \pi \circ \Phi : V_\mathfrak{g} \rightarrow G \rightarrow G/K.
\end{equation}
Note that if $K = \{e\}$, then $\Phi_K = \Phi$.  Note also that $\Phi_K$ has the same properties as in Propositions \ref{prop1} and \ref{prop2}.

In the rest of this section, we mention several facts which will be used later. By (\ref{equiv4}) (i), the following  diagram commutes for each $g \in P(G, G \times \{e\})$.
\begin{equation} \label{commute1}\begin{CD}
V_\mathfrak{g}@>g* >> V_\mathfrak{g}
\\
@V\Phi VV @V\Phi VV
\\
G @> (g(0),\, e) >> G
\end{CD}
\end{equation}

Let $G$, $K$ be as above. For $a \in G$ we denote by $l_a$ the left translation by $a$ and $L_a$ an isometry on $G/K$ defined by $L_a(bK) := abK$ for $b \in G$. Then a diagram
\begin{equation}\label{commute2}
\begin{CD}
G @>l_a>> G
\\
@V\pi VV @V\pi VV
\\
G/K @>L_a>> G/K
\end{CD}
\end{equation}
commutes. Combining (\ref{commute1}) with (\ref{commute2}),  the following diagram  commutes for  $g \in P(G, G \times \{e\})$ and $a := g(0)$.
\begin{equation} \label{commute3}
\begin{CD}
V_\mathfrak{g} @>g*>> V_\mathfrak{g}
\\
@V \Phi_K VV  @V\Phi_{K} VV
\\
G/K @> L_a >> G/K
\end{CD}
\end{equation}

Let $G$, $H$ be as above. For each $a \in G$, set $H^a := (a,e)^{-1} H (a, e)$. We have $H \cdot a = l_a(H^a \cdot e)$. Then it follows from (\ref{equiv4}) (ii) and (\ref{commute1}) that for $g \in P(G,  G \times \{e\})$, $u := g *\hat{0}$ and $a := \Phi(u) = g(0)$,
\begin{equation*}
P(G,H) * u =  g* (P(G, H^{a} ) * \hat{0}).
\end{equation*}

The following  are Lie algebraic expressions of the tangent spaces of orbits.  Since $\left. \frac{d}{ds}\right|_{s=0} \left(\exp sZ \right) * \hat{0} = -Z'$ for $Z \in \operatorname{Lie}  P(G,H)$, we have 
\begin{equation}\label{tang}
\begin{array}{l}
T_{\hat{0}} (P(G, H) * \hat{0}) = \{-Z' \in T_{\hat{0}} V_\mathfrak{g} \mid Z \in \operatorname{Lie} P(G,H) \}, 
\vspace{1mm}\\
T_e (H \cdot e) = \{x - y \in \mathfrak{g}\mid (x,y) \in \mathfrak{h} \}.
\end{array}
\end{equation}

%%%%%%%%%%%%%%%%%%%%%%%%
\section{Second fundamental forms and shape operators}\label{sec2}
In this section we study the second fundamental form and the shape operator of a PF submanifold  obtained through the parallel transport map.

Let $G$, $V_\mathfrak{g}$, $\Phi$ be as in Section $\ref{prel}$. Let $F := \Phi^{-1}(e)$ be a fiber of $\Phi$ at $e \in G$.  Denote by $\iota : F \rightarrow V_\mathfrak{g}$  the inclusion map and regard $F$ as a submanifold $V_\mathfrak{g}$. Recall that $P(G, \{e\} \times \{e\})$ acts on $F$ transitively and freely. Let  
$
\mathcal{E} : T_{\hat{0}} V_\mathfrak{g} \rightarrow \Gamma (\iota^*TV_\mathfrak{g})
$
denote a map of the extension to a $P(G, \{e\} \times \{e\})$-equivariant vector filed along $F$, i.e., 
\begin{equation}\label{equiv}
\mathcal{E}(X)_{g * \hat{0} } := g  X g ^{-1},
\quad
X \in T_{\hat{0}}  V_\mathfrak{g}, \ g  \in P(G,\{e\} \times \{e\}).
\end{equation}
By (\ref{tang}) we have 
\begin{equation} \label{fibertang}
T_{\hat{0}} F = \{-Q' \in V_\mathfrak{g} \mid Q \in H^1([0,1], \mathfrak{g}), \ Q(0) = Q(1) = 0 \},
\end{equation}
and by (\ref{decomp0}) we have $T^\perp_{\hat{0}} F =\hat{\mathfrak{g}}$. 

\begin{lem}\label{fiber1}
The Levi-Civita connection $\nabla^{TF}$, the second fundamental form $\alpha^{F}$, the shape operator $A^{F}$, and the normal connection $\nabla^{T^\perp F}$ of $F$ satisfy the following. For $-Q', -R' \in T_{\hat{0}} F$, $\hat{\xi} \in T^\perp_{\hat{0}} F$, 
\begin{align*}
\textup{(i)}&& &{\textstyle \nabla^{TF}_{-Q'} \, \mathcal{E}(-R')
=
[Q , -R'] -  \int_0^1 [Q, -R'](t) dt},&
\\
\textup{(ii)}&& &{\textstyle \alpha^{F}(-Q',-R')
=
\int_0^1[Q, -R'](t)  dt },&
\\
\textup{(iii)}&& &{\textstyle A^{F}_{\hat{\xi}}(-Q') 
=
-[Q, \hat{\xi}] + \left[\int_0^1 Q(t) dt, \xi\right]}, &
\\
\textup{(iv)}&& &{\textstyle \nabla^{T^\perp F}_{-Q'}  \mathcal{E} (\hat{\xi})
=
\left[\int_0^1 Q(t) dt , \xi \right]}.&
\end{align*}
\end{lem}

\begin{proof} 
Since $V_\mathfrak{g}$ is flat, it follows from (\ref{equiv}) that
\begin{align*}
&{\textstyle \nabla^{\iota^*TV_\mathfrak{g}}_{-Q'} \mathcal{E}(-R')
=
\left. \frac{d}{ds}\right|_{s=0} \mathcal{E}(-R')_{\left(\exp sQ \right) * \hat{0}}
=
[Q, -R']},
\\
&{\textstyle \nabla^{\iota^*TV_\mathfrak{g}}_{-Q'} \mathcal{E}(\hat{\xi})
=
\left. \frac{d}{ds}\right|_{s=0} \mathcal{E}(\hat{\xi})_{(\exp sQ) * \hat{0}}
=
[Q, \hat{\xi}]}.
\end{align*}
By (\ref{decomp0}) our claim follows.
\end{proof}

The following theorem gives Lie algebraic formulas for the second fundamental form and the shape operator of a PF submanifold obtained through $\Phi$.
\begin{thm} \label{preimage1}
Let $N$ be a closed submanifold of $G$ through $e \in G$. Denote respectively by $\alpha^N$ and $A^N$ the second fundamental form and the  shape operator of $N$,  and  by $\alpha^{\Phi^{-1}(N)}$ and   $A^{\Phi^{-1}(N)}$ those of $\Phi^{-1}(N)$. For $X, Y \in T_{\hat{0}} \Phi^{-1}(N)$, $\hat{\xi} \in T^\perp _{\hat{0}} \Phi^{-1}(N) (\subset \hat{\mathfrak{g}})$, 
\begin{align*}
\textup{(i)} &&& {\textstyle \alpha^{\Phi^{-1}(N)}(X,Y)
=
\alpha^{N} \left(\int_0^1 X(t) dt, \ \int_0^1 Y(t) dt\right)}
\\
&&& \quad {\textstyle
+
\frac{1}{2} \left[\int_0^1 X(t) dt, \int_0^1 Y(t) dt \right]^\perp
-
\left(\int_0^1 \left[\int_0^t X(s) ds, Y(t)\right] dt \right)^\perp,
}
\\
\textup{(ii)}&& &{\textstyle
A^{\Phi^{-1}(N)}_{\hat{\xi}}(X) (t)
=
A^N_\xi \left(\int_0^1 X(t) dt \right)
}
\\
&&& \quad
{\textstyle
-
\frac{1}{2}\left[\int_0^1 X(t) dt, \xi\right]^\top
+
\left[\int_0^t X(s) ds , \xi\right]
-
\left[\int_0^1\int_0^t X(s) ds  dt , \xi\right]^\perp,
}
\end{align*}
where $\top$ and $\perp$ denote the    projections of $\mathfrak{g}$ onto $T_eN$ and $T_e^\perp N$, respectively.
\end{thm}
\begin{proof}
(i) Recall that $\Phi$ is a Riemannian submersion with decomposition (\ref{decomp0}). We use superscripts $h$ and $v$ to denote the projections of $T_{\hat{0}} V_\mathfrak{g}$ onto $\hat{\mathfrak{g}}$ and $T_{\hat{0}} F$, respectively. Set $\bar{N} := \Phi^{-1}(N)$. Then  
\begin{align*}
\alpha^{\bar{N}}(X, Y)
&
=
\alpha^{\bar{N}}(X^h, Y^h)
+
\alpha^{\bar{N}}(X^h, Y^v)
+
\alpha^{\bar{N}}(X^v, Y^h)
+
\alpha^{\bar{N}}(X^v, Y^v)
\\
& 
=
\alpha^{N} (d \Phi (X), d \Phi (Y))
\\
&
\quad 
+
(\nabla^{T^\perp F}_{Y^v} \mathcal{E}(X^h))_{T^\perp_{\hat{0}} \bar{N}}
+
(\nabla^{T^\perp F}_{X^v} \mathcal{E}(Y^h))_{T^\perp_{\hat{0}} \bar{N}}
+
\alpha^F (X^v, Y^v)_{T^\perp_{\hat{0}} \bar{N}}.
\end{align*}
Define $Q, R \in H^1([0,1], \mathfrak{g})$  by 
\begin{equation*}
\left\{
\begin{array}{l}
X^v = -Q',
\\
Q(0) = Q(1) =0,
\end{array}
\right.
\qquad 
\left\{
\begin{array}{l}
Y^v = -R',
\\
Y(0) = Y(1) =0.
\end{array}
\right. 
\end{equation*}
Explicitly $Q$ and $R$ are
\begin{equation*}
{\textstyle
Q= t X^h - \int_0^t X(s) ds,} 
\qquad
{\textstyle
R=  t Y^h - \int_0^t Y(s) ds.}
\end{equation*}
By Lemma \ref{fiber1} we have 
\begin{align*}
&
\alpha^{\bar{N}}(X,Y) - \alpha^{N}(d \Phi (X), d \Phi (Y))
\\
&
{\textstyle
=
\left[\int_0^1R(t) dt , X^h\right]^\perp
+
\left[\int_0^1 Q(t) dt , Y^h \right]^\perp
+
\left(\int_0^1 [Q, -R'](t) dt\right)^\perp
}.
\end{align*}
Let us calculate each term above. 
\begin{align*}
&
{\textstyle
\left[\int_0^1R(t) dt , X^h \right]
=
\frac{1}{2}[Y^h, X^h]
-
\left[\int_0^1 \int_0^t Y(s)ds dt, X^h\right]
},
\\
&
{\textstyle
\left[\int_0^1 Q(t) dt , Y^h \right]
=
\frac{1}{2} [X^h, Y^h]
- \left[\int_0^1 \int_0^t X(s) ds dt , Y^h \right]
}.
\end{align*}
For the third term, note that integrating by parts we have 
\begin{equation*}
{\textstyle
\int_0^1 t Y(t) dt
=
\left[t \int_0^t Y(s) ds \right]^1_0 - \int_0^1 \int_0^t Y(s) ds dt
=
Y^h  - \int_0^1 \int_0^t Y(s) ds dt}.
\end{equation*}
Using this we have 
\begin{align*}
{\textstyle
\int_0^1 [Q, -R'](t) dt  
}
&
{\textstyle
= \int_0^1 \left[t X^h - \int_0^t X(s) ds, Y(t) -Y^h \right] dt 
}
\\
& {\textstyle=
\left[X^h, \int_0^1 t Y(t) dt \right]
-
\frac{1}{2}[X^h, Y^h]
}
\\
&\quad {\textstyle
-
\int_0^1 \left[\int_0^t X(s) ds, Y(t)\right] dt 
+
\left[\int_0^1 \int_0^t X(s) ds dt, Y^h \right]}
\\
&
{\textstyle
= \frac{1}{2} \left[X^h, Y^h \right]
-
\left[X^h, \int_0^1 \int_0^t Y(s)ds dt \right]
}
\\
& \quad {\textstyle
-
\int_0^1 \left[\int_0^t X(s) ds, Y(t) \right] dt 
+
\left[\int_0^1 \int_0^t X(s) ds dt, Y^h\right]}.
\end{align*}
From these calculations we obtain  (i).  

(ii) By (i) and $\operatorname{Ad}(G)$-invariance of the inner product of $\mathfrak{g}$, we have
\begin{align*}
{\textstyle
\langle 
A^{\bar{N}}_{\hat{\xi}}(X), Y \rangle_{L^2}
}
&
{\textstyle
=
\langle 
\alpha^{\bar{N}} (X,Y), \hat{\xi} \rangle_{L^2}
}
\\
&
{\textstyle
=
\left\langle 
A^N_{\xi}(X^h) - \frac{1}{2} [X^h, \xi] + \left[ \int_0^t X(s) ds, \xi \right], Y
\right\rangle_{L^2}}.
\end{align*}
This proves (ii).
\end{proof}
\begin{rem}\label{preimage3} 
Let $G$, $K$, $\pi$, $\Phi_K$ be as in Section $\ref{prel}$. Let $N$ be a closed submanifold of $G/K$ through $e K \in G/K$. It follows that   for $x, y \in T_e \pi^{-1}(N)$, $\xi \in T_{eK}^\perp N \cong T_e ^\perp \pi^{-1}(N)$,
\begin{align*}
\textup{(i)}&& & 
\alpha^{\pi^{-1}(N)}(x, y)
=
\alpha^N(x_{\mathfrak{m}}, y_{\mathfrak{m}})
-
\frac{1}{2} [x_{\mathfrak{k}}, y_{\mathfrak{m}}]^\perp
+
\frac{1}{2} [x_{\mathfrak{m}}, y_{\mathfrak{k}}]^\perp,&
\\
\textup{(ii)}&& & 
A_{\xi}^{\pi^{-1}(N)}(x)
=
A^N_\xi(x_{\mathfrak{m}}) - \frac{1}{2}[x_{\mathfrak{m}}, \xi]_{\mathfrak{k}} + \frac{1}{2} [x_{\mathfrak{k}}, \xi]^\top,&
\end{align*} 
where $\alpha^{N}$ and  $A^{N}$ are respectively the second fundamental form and shape operator of $N$, and $\alpha^{\pi^{-1}(N)}$, $A^{\pi^{-1}(N)}$ are those of $\pi^{-1}(N)$. By using these formulas we can easily generalize Theorem \ref{preimage1} to the case of the parallel transport map $\Phi_K$ over $G/K$. 
\end{rem}
The following corollary can be obtained easily from Theorem \ref{preimage1} (ii). 
\begin{cor}\label{preimage2}
Let $N$ be as in Theorem $\ref{preimage1}$. Decompose $T_{\hat{0}} {\Phi^{-1}(N)} = T_eN \oplus T_{\hat{0}} F$.  For $\xi \in T_e^\perp N$, $x \in T_e N$, $-Q' \in T_{\hat{0}} F$ with expression $(\ref{fibertang})$,
\begin{align*}
\textup{(i)}&&& {\textstyle A^{{\Phi^{-1}(N)}}_{\hat{\xi}}(\hat{x})(t)}
 {\textstyle = 
A^N_\xi(x) + \left(t- \frac{1}{2}\right)[x, \xi]},&
\\
\textup{(ii)}&&& {\textstyle
A^{{\Phi^{-1}(N)}}_{\hat{\xi}}(-Q') }
{\textstyle =
- [Q , \hat{\xi}] + \left[\int_0^1 Q(t) dt, \xi \right]^\perp}. &
\end{align*}
\end{cor}

Here we mention second fundamental forms and shape operators of $P(G,H)$-orbits. The following formulas generalize Lemma \ref{fiber1} (ii) and (iii). Recall the Lie algebraic expressions of the tangent spaces (\ref{tang}).

\begin{thm}\label{thmC} 
Let $H$ be as in Section $\ref{prel}$. The second fundamental form $\alpha^{P(G,H) * \hat{0}}$ and the shape operator $A^{P(G,H) * \hat{0}}$ of an orbit $P(G,H) * \hat{0}$ through $\hat{0} \in V_\mathfrak{g}$  are  given by the following. For $-Z' , -W' \in T_{\hat{0}} (P(G,H) * \hat{0})$, $\hat{\xi} \in T_{\hat{0}}^\perp (P(G,H) * \hat{0})$, 
\begin{align*}
\textup{(i)}& &&{\textstyle
\alpha^{P(G,H) * \hat{0}}(-Z', -W')
=
\int_0^1 \{[Z, -W'](t)\}^\perp dt,}
&
\\
\textup{(ii)}& &&{\textstyle
A^{P(G,H)*\hat{0}}_{\hat{\xi}}(-Z') 
=
-[Z, \hat{\xi}] + \left[\int_0^1 Z(t) dt, \xi\right]^\perp,}&
\end{align*}
where $\top$ and $\perp$ denote the   projections of $\mathfrak{g}$  onto $T_e(H \cdot e)$ and $T^\perp_e(H \cdot e)$,  respectively. 
\end{thm}
In order to prove Theorem \ref{thmC},  we use the following formulas for $H$-orbits. These formulas can be proved independently by straightforward computations.  
\begin{prop}\label{thmB}
Let $H$ be as in Section $\ref{prel}$. 
The second fundamental form $\alpha^{H \cdot e}$ and the shape operator $A^{H \cdot e}$ of an orbit $H \cdot e$ through $e \in G$ are given by the following.  For $x-y, z-w \in T_e (H \cdot e)$, $\xi \in T^\perp_e (H \cdot e)$, 
\begin{align*}
\textup{(i)}& & &\alpha^{H \cdot e}(x- y, z-w) 
=
-\frac{1}{2}[x- y, \ z + w]^\perp
=-\frac{1}{2}\left([x, w] - [y, z]\right)^\perp,
\\
\textup{(ii)}& &&
A^{H \cdot e}_{\xi}(x-y) = - \frac{1}{2}[x + y, \xi]^\top.\ 
\end{align*}
\end{prop}

\begin{proof}[Proof of Theorem $\ref{thmC}$]
Set $N := H \cdot e$ and $\bar{N} := P(G,H) * \hat{0}$ so that $\bar{N} = \Phi^{-1}(N)$. By Theorem \ref{preimage1} (i), Proposition \ref{thmB} (i) and the fact that $\alpha^{N}$ is a symmetric bilinear form, we have 
\begin{align*}
{\textstyle
\alpha^{\bar{N}}(-Z', -W')}
&{\textstyle
=
\alpha^{N}(W(0) - W(1), Z(0)- Z(1))
+
\frac{1}{2} \left[Z(0) - Z(1),W(0) - W(1)\right]^\perp
}
\\
&{\textstyle \quad
-
\left(\int_0^1 \left[Z(0)-Z(t), -W'(t)\right] dt \right)^\perp}
\\
&{\textstyle
=
- \frac{1}{2} [W(0) - W(1), Z(0) + Z(1)]^\perp
+
\frac{1}{2} \left[Z(0)- Z(1),W(0) - W(1)\right]^\perp
}
\\
&{\textstyle \quad -
\left(
\left[Z(0), \int_0^1  -W'(t)dt \right] 
-
\int_0^1 \left[Z, -W'\right](t) dt 
\right)^\perp}
\\
&{\textstyle
=
\left(\int_0^1 \left[Z, -W'\right](t) dt \right)^\perp.}
\end{align*}
This proves (i). (ii) follows  from Theorem \ref{preimage1} (ii) and Proposition \ref{thmB} (ii). 
\end{proof}

%%%%%%%%%%%%%%%%%%%%%%%%
\section{Totally geodesic properties}\label{sec3} 

The purpose of this section is to give  criteria for a PF submanifold $\Phi^{-1}(N)$ to be totally geodesic (Theorem \ref{tgeod8}), where $\Phi$ is the parallel transport map and  $N$ is a closed connected submanifold of $G$ through $e \in G$.  From these criteria we see that $\Phi^{-1}(N)$ is \emph{not} totally geodesic except for rare cases. This leads us to a remarkable property of  homogeneous minimal submanifolds in Hilbert spaces (Remark \ref{homgmin}). 

Let $\mathfrak{g}^{ss} = [\mathfrak{g}, \mathfrak{g}]$ denote the semisimple part  and $\mathfrak{c}(\mathfrak{g})$ the center of $\mathfrak{g}$. We know the orthogonal direct sum decomposition $\mathfrak{g} = \mathfrak{g}^{ss} \oplus \mathfrak{c} (\mathfrak{g})$. We write $G^{ss}$ for  a connected subgroup of $G$ generated by $\mathfrak{g}^{ss}$. 

\begin{thm}\label{tgeod8}
Let $G$, $V_\mathfrak{g}$, $\Phi$ be as in Section $\ref{prel}$ and $N$ a closed connected  submanifold of $G$ through $e \in G$. The following are equivalent.
\begin{enumerate}
\item[\textup{(i)}] $\Phi^{-1}(N)$ is a totally geodesic PF submanifold of $V_\mathfrak{g}$.

\item[\textup{(ii)}] $N$ is a closed subgroup of $G$ such that $\mathfrak{g}^{ss} \subset T_e N$.

\item[\textup{(iii)}] $N$ is a closed subgroup of $G$ such that $T^\perp_e N \subset \mathfrak{c}(\mathfrak{g})$.

\item[\textup{(iv)}] $N$ is a closed subgroup of $G$  which contains $G^{ss}$. 
\end{enumerate}
\end{thm}

\begin{proof}
Equivalence of (ii), (iii) and (iv) is clear. (iii) $\Rightarrow$ (i): Since $N$ is totally geodesic and $T^\perp _e N \subset \mathfrak{c}(\mathfrak{g})$,  it follows from Theorem \ref{preimage1} (ii) that $\Phi^{-1}(N)$ is totally geodesic at $\hat{0} \in V_\mathfrak{g}$. Since $N$ is a closed subgroup of $G$,  we have $\Phi^{-1}(N) = \Phi^{-1}((\{e\} \times N) \cdot e) = P(G, e \times N) * \hat{0}$ and in particular $\Phi^{-1}(N)$ is homogeneous. Thus $\Phi^{-1}(N)$ is a totally geodesic PF submanifold of $V_\mathfrak{g}$.  (i) $\Rightarrow$ (iii): Let $\xi \in T^\perp_{e} N$ and  $x \in \mathfrak{g}$. Since $\Phi$ is a Riemannian submersion, $N$ is totally geodesic. Thus by Corollary \ref{preimage2} (i) we have
\begin{equation*}
{\textstyle 
0 
= 
A^{\Phi^{-1}(N)}_{\hat{\xi}}(\hat{x})(t) 
=
(t - \frac{1}{2})[x, \xi].}
\end{equation*}
for all $t \in [0,1]$. This shows  
$[x, \xi] = 0$ and thus we obtain $T^\perp_e N \subset  \mathfrak{c}(\mathfrak{g})$, which is equivalent to $\mathfrak{g}^{ss} \subset T_eN$. Then $T_e N$ is a Lie subalgebra of $\mathfrak{g}$ because $\mathfrak{g}^{ss} = [\mathfrak{g}, \mathfrak{g}]$. Since $N$ is connected and totally geodesic,   $N$ is identical to a connected Lie subgroup of $G$ generated by $T_e N$. Hence $N$ is a closed subgroup of $G$ and (iii) follows.  
\end{proof}

\begin{cor} \ \label{tgeod7}
\begin{enumerate}
\item[\textup{(i)}] If $G$ is abelian, then $\Phi^{-1}(N)$ is a totally geodesic submanifold of $V_\mathfrak{g}$ for any closed connected submanifold $N$ of $G$. 

\item[\textup{(ii)}] Suppose that $G$ is semisimple. Let $N$ be a closed connected submanifold of $G$. Then the following are equivalent. \textup{(a)} $\Phi^{-1}(N)$ is a totally geodesic submanifold of $V_\mathfrak{g}$.  \textup{(b)} $N = G$.
\textup{(c)} $\Phi^{-1}(N)= V_\mathfrak{g}$.
\end{enumerate}
\end{cor}
\begin{proof}
(i) Choose $a \in N$ and set $N' : = a^{-1}N$. Then $\Phi^{-1}(N')$ is totally geodesic and thus the assertion follows from commutativity of (\ref{commute1}). (ii) is clear. 
\end{proof}

\begin{rem} \label{homgmin} 
It is known that any homogeneous minimal  submanifold in a finite dimensional Euclidean space must be totally geodesic (\cite{Sca02}). From  Theorem \ref{tgeod8}  and examples of homogeneous weakly reflective PF submanifolds given Sections 5, 6 and 7, we see that \emph{in infinite dimensional Hilbert spaces, there exists so many homogeneous minimal submanifolds  which are \underline{not} totally geodesic}.
\end{rem}

For fibers of the parallel transport map, we have the following. 
\begin{cor} \ \label{tgeod1}
\begin{enumerate}
\item[\textup{(i)}] Let $G$, $K$, $\Phi_K$ be as in Section $\ref{prel}$. The following are equivalent. \textup{(a)} The fiber of $\Phi_K$ at $eK \in G/K$ is a totally geodesic submanifold of $V_\mathfrak{g}$. \textup{(b)} Each fiber of $\Phi_K$ is a totally geodesic submanifold of $V_\mathfrak{g}$. \textup{(c)} $\mathfrak{g}^{ss} \subset \mathfrak{k}$.  \textup{(d)} $\mathfrak{m} \subset \mathfrak{c}(\mathfrak{g})$. \textup{(e)} $G^{ss} \subset K$.

\item[\textup{(ii)}] Let $G$, $\Phi$ be as in Section $\ref{prel}$. The following are equivalent. \textup{(a)} The fiber of $\Phi$ at $e \in G$ is a totally geodesic submanifold of $V_\mathfrak{g}$. \textup{(b)} Each fiber of $\Phi$ is a totally geodesic submanifold of $V_\mathfrak{g}$.  \textup{(c)} $G$ is  a torus.
\end{enumerate}
\end{cor}

\begin{rem}
Recall that $\Phi: V_\mathfrak{g}\rightarrow G$ is a principal $P(G, \{e\} \times \{e\})$-bundle which is not trivial in general. Corollary \ref{tgeod1} (ii) shows that  $\Phi$ is a Hilbert space bundle if and only if $G$ is a torus.   In this case  $\Phi$ is a trivial bundle.
This  agrees with Kuiper's theorem (\cite[p.\ 67]{BB85}), stating that any  Hilbert space bundle must be trivial. 
\end{rem}

%%%%%%%%%%%%%%%%%%%%%%%%
\section{The canonical reflection of the Hilbert space $V_\mathfrak{g}$}
In this section we focus on intrinsic symmetry of the parallel transport map and show that each fiber of the parallel transport map is weakly reflective.

Let $G$, $\mathcal{G}$, $V_\mathfrak{g}$, $\Phi$ be as in Section $\ref{prel}$. Denote  by ${}_\#$ a map which corresponds to each $u \in V_\mathfrak{g}$ (resp.\  $g \in\mathcal{G}$) the \emph{inverse path} $u_\#$ (resp.\  $g_\# \in \mathcal{G}$):
\begin{equation*}
u_\#(t) :=u(1-t), 
\qquad
g_\#(t) := g(1-t).
\end{equation*}

\begin{defn} 
The \emph{canonical reflection} $\mathfrak{r}$ of $V_\mathfrak{g}$ is an involutive linear orthogonal transformation of $V_\mathfrak{g}$ defined by 
\begin{equation*}
\mathfrak{r}(u) := - u_\#,
\quad
u \in V_\mathfrak{g}. 
\end{equation*}
\end{defn}

Since $(g_\#)' = -(g')_\#$ for each $g \in \mathcal{G}$, we have 
\begin{equation*}
\mathfrak{r}(g * \hat{0}) = g_\# * \hat{0}, \quad g \in \mathcal{G}.
\end{equation*}
Thus by (\ref{equiv4}) (i), we obtain a commutative diagram
\begin{equation} \label{canocom}
\begin{CD}
V_\mathfrak{g}@>\mathfrak{r}>> V_\mathfrak{g}
\\
@V\Phi VV @V\Phi VV
\\
G @>\mathfrak{i}>> G
\end{CD}
\end{equation}
where $\mathfrak{i}$ is an isometry of $G$ defined by $\mathfrak{i}(a) = a^{-1}$ for each $a \in G$. It also follows that  the following diagram commutes.
\begin{equation} \label{canocom2}
\begin{array}{ccccc}
\mathcal{G}& \supset &P(G, \{e\} \times G) &\curvearrowright& V_\mathfrak{g} \ \vspace{1.5mm}
\\
{}_\# \downarrow\ \ &&{}_\# \downarrow\ \ &&\mathfrak{r} \downarrow\ \  \ \vspace{1.5mm}
\\
\mathcal{G}& \supset&P(G, G \times \{e\}) &\curvearrowright& V_\mathfrak{g}
\end{array}
\end{equation}

For each $g \in P(G, \{e\} \times G)$ we can easily see that 
\begin{center}
$g_\# g(1)^{-1} \in P(G, \{e\} \times G)$\quad and \quad 
$
((g_{\#}) g(1)^{-1}) * \hat{0}
=
g_\# * \hat{0}
$.
\end{center}
Hence via an isometry (\ref{isometry}), $\mathfrak{r}$ induces an involutive isometry $\tilde{\mathfrak{r}}$ of $P(G, \{e\} \times G)$, which is defined by 
\begin{equation*} 
\tilde{\mathfrak{r}}(g) = g(1)^{-1} g_\#,
\quad
g \in P(G, \{e\} \times G).
\end{equation*}

The reflective submanifold associated to $\mathfrak{r}$ is described as follows. 
\begin{prop}
Let $W$ denote the fixed point set of $\mathfrak{r}$.  Then 
\begin{enumerate}
\item[\textup{(i)}] $W$ is a closed linear subspace of $V_\mathfrak{g}$,

\item[\textup{(ii)}] $W$ is isomorphic to the Hilbert space $H^0([0,1/2], \mathfrak{g})$,

\item[\textup{(iii)}] $W$ is contained in the fiber of $\Phi$ at $e \in G$.
\end{enumerate}
\end{prop}
\begin{proof}
(i) follows from linearity of $\mathfrak{r}$. (ii) is clear by the expression 
$
W = \{u \in V_\mathfrak{g}  \mid \forall t \in [0,1], \ u(t) = - u(1-t)\}.
$
(iii) follows from (\ref{canocom}). 
\end{proof}

One application of the canonical reflection is the following.
\begin{thm}\label{austere} Let $N$ be a closed totally geodesic submanifold of $G$. Then $\Phi^{-1}(N)$ is an austere PF submanifold of $V_\mathfrak{g}$. 
\end{thm}
\begin{proof} 
Let $(u, X) \in T^\perp \Phi^{-1}(N)$. Denote by $A_X$ the corresponding shape operator of $\Phi^{-1}(N)$. Choose $g \in P(G, G \times \{e\})$ so that $u = g * \hat{0}$. Set $a := \Phi(u) = g(0)$, $N' := a^{-1}N$ and $\xi := a^{-1}(d \Phi(X)) \in T_e^\perp N'$. The horizontal lift of $\xi$ at $\hat{0} \in V_\mathfrak{g}$ is the constant path $\hat{\xi} \in T^\perp_{\hat{0}} \Phi^{-1}(N)$. By commutativity of (\ref{commute1}) we have $g * (\Phi^{-1}(N')) = \Phi^{-1}(N)$ and $(d g*) \hat{\xi} = X$. Thus in order to show the invariance of the set of eigenvalues of $A_X$ by the multiplication by $(-1)$,  it suffices to prove this for the shape operator $A_{\hat{\xi}}$ of $\Phi^{-1}(N')$ in the direction of $\hat{\xi}$. Since $N'$ is also totally geodesic, it follows from Corollary \ref{preimage2} that the diagram 
\begin{equation*}
\begin{CD}
T_{\hat{0}} \Phi^{-1}(N') @>A_{\hat{\xi}}>> T_{\hat{0}} \Phi^{-1}(N')
\\
@V \mathfrak{r} VV @V \mathfrak{r} VV
\\
T_{\hat{0}} \Phi^{-1}(N') @>-A_{\hat{\xi}} >> T_{\hat{0}} \Phi^{-1}(N')
\end{CD}
\end{equation*}
commutes. This implies that the set of eigenvalues of $A_{\hat{\xi}}$ is invariant under the multiplication by $(-1)$. Thus our claim follows.
\end{proof}

\begin{cor}\label{austere2} Let $G$, $H$ be as in Section $\ref{prel}$.  If an orbit $H \cdot a$ through $a \in G$ is totally geodesic submanifold of $G$, then  the orbit $P(G, H) * u$ through $ u \in \Phi^{-1}(a)$ is an austere PF  submanifold of $V_\mathfrak{g}$. 
\end{cor}

For the study of weakly reflective submanifolds later, we now introduce the following lemma.

\begin{lem} \label{weaksubm}
Let $\mathcal{M}$ and $\mathcal{B}$ be Riemannian Hilbert manifolds and 
$\pi :\mathcal{M} \rightarrow \mathcal{B}$ be a Riemannian submersion. Let $N$ be a closed  submanifold of $\mathcal{B}$ and  $(p, \xi) \in T^\perp \pi^{-1}(N)$.
Suppose that $\nu_\xi$ and $\nu_{d \pi (\xi)}$ are isometries of 
$\mathcal{M}$ and $\mathcal{B}$, respectively. Suppose also that $\nu_\xi(p) =p$, $\nu_{d\pi(\xi)}(\pi(p)) = \pi(p)$ and the following diagram commutes. 
\begin{equation} \label{commute5}
\begin{CD}
\mathcal{M} @>\nu_{\xi}>> \mathcal{M}
\\
@V \pi VV @V \pi VV
\\
\mathcal{B} @>\nu_{d\pi(\xi)}>> \mathcal{B}
\end{CD}
\end{equation} 
Then the following are equivalent:

\begin{enumerate}
\item[\textup{(i)}] $\nu_\xi$ is a reflection of $\pi^{-1}(N)$ with respect to $\xi$,

\item[\textup{(ii)}] $\nu_{d \pi (\xi)}$ is a reflection of $N$ with respect to $d \pi(\xi)$.
\end{enumerate}
\end{lem}
\begin{proof}
It is easy to see that the condition $\nu_\xi(\pi^{-1}(N)) = \pi^{-1}(N)$ is equivalent to the condition $\nu_{d \pi (\xi)}(N) = N$. Then by commutativity of the diagram
\begin{equation*}
\begin{CD}
T^\perp_p \pi^{-1} (N) @>d\nu_{\xi}>> T^\perp_p \pi^{-1} (N)
\\
@V d\pi VV @V d\pi VV
\\
T^\perp_{\pi(p)} N @>d\nu_{d\pi(\xi)}>> T^\perp_{\pi(p)} N
\end{CD}
\end{equation*}
$d\nu_{\xi}(\xi) = - \xi$ if and only if $d \nu_{d \pi (\xi)}(d \pi(\xi)) =- d \pi (\xi)$. This proves the lemma.
\end{proof}

Another application of the canonical reflection is the following.
\begin{thm}\label{type1} 
Let $G$, $H$ be as in Section $\ref{prel}$. Suppose that an orbit $H \cdot e$  through $e \in G$ satisfies the condition  $(H \cdot e)^{-1} = H \cdot e$. Then 
\begin{enumerate}

\item[\textup{(i)}] $H \cdot e$ is a weakly reflective submanifold of $G$,
\item[\textup{(ii)}] $P(G, H) * \hat{0}$ is a weakly reflective PF submanifold of $V_\mathfrak{g}$. 
\end{enumerate}
\end{thm}
\begin{proof}
(i) It is easy to see that $\mathfrak{i}$ is a reflection of $H \cdot e$ with respect to any normal vector at $e \in G$. By homogeneity, $H \cdot e$ is a weakly reflective submanifold of $G$.  (ii) By (\ref{canocom}) and Lemma \ref{weaksubm}, $\mathfrak{r}$ is a reflection of $P(G,H) * \hat{0}$
with respect to any normal vector at $\hat{0}$. Since $P(G,H) * \hat{0}$ is homogeneous, our claim follows.
\end{proof}

A typical example of $H$ satisfying the condition $(H \cdot e)^{-1} = H \cdot e$ is that $H = \{e\} \times K$ or $K \times \{e\}$, where $K$ is a closed subgroup of $G$. The following is another example such that $H \cdot e$ is \emph{not} a subgroup of $G$.
\begin{example}\label{examp0}
For each automorphism $\sigma$ of $G$, $G(\sigma) := \{(a, \sigma(a)) \mid a \in G\}$ is a closed subgroup of $G \times G$. The $G(\sigma)$-action on $G$ defined by (\ref{Haction}) is called the \emph{Conlon's $\sigma$-action} (\cite{Con64}). From now on we suppose that $\sigma^2 = \operatorname{id}$. It easily follows that $H := G(\sigma)$ satisfies $(H \cdot e)^{-1} = H \cdot e$. Thus by Theorem \ref{type1},  $G(\sigma) \cdot e$ is a weakly reflective submanifold of $G$, and $P(G, G(\sigma)) * \hat{0}$ is a weakly reflective PF submanifold of $V_\mathfrak{g}$. Note that  $G(\sigma) \cdot e$ is not a subgroup of $G$ in general. Note also that $G(\sigma) \cdot e$ is a totally geodesic submanifold of $G$ since it is given by the Cartan immersion (\cite[p.\ 347]{Hel01}) $G/K \rightarrow G$, $aK \mapsto a \sigma(a)^{-1}$, where $K$ is the fixed point set of $\sigma$. On the other hand $P(G, G(\sigma)) * \hat{0}$ is not  totally geodesic in most cases by Theorem \ref{tgeod8}. 
\end{example}

It was essentially  proved (\cite[Theorem 4.11]{KT93}, \cite[Corollary 6.3]{HLO06}) that each fiber of the parallel transport map is  an austere PF submanifold of $V_\mathfrak{g}$.  The following corollary asserts that the fibers have higher symmetry.

\begin{cor} \label{weakfiber}
Let $G$, $V_\mathfrak{g}$, $K$, $\Phi_K$ be as in Section $\ref{prel}$. Each fiber of $\Phi_K: V_\mathfrak{g}  \rightarrow G/K$ is a weakly reflective PF submanifold of $V_\mathfrak{g}$. 
\end{cor}
\begin{proof}
By Theorem \ref{type1} the fiber of $\Phi_K$ at $eK$ is a weakly reflective PF submanifold of $V_\mathfrak{g}$. Since any two fibers of $\Phi_K$ are congruent under the isometry on $V_\mathfrak{g}$, each fiber of $\Phi_K$ is a weakly reflective PF submanifold of $V_\mathfrak{g}$.
\end{proof}

\begin{rem}
Corollary \ref{weakfiber} shows that for each $a \in G$ and each closed subgroup $K$ of $G$, the inverse image $\Phi^{-1}(aK)$ is  weakly reflective. This should be also  compared  with Theorem \ref{austere}, stating that if $N$ is totally geodesic then $\Phi^{-1}(N)$ is austere.
\end{rem}

\begin{cor}
Let $G$, $V_\mathfrak{g}$, $\Phi$ be as in Section $\ref{prel}$. Each fiber of $\Phi : V_\mathfrak{g} \rightarrow G$ is a weakly reflective PF submanifold of $V_\mathfrak{g}$.
\end{cor}

%%%%%%%%%%%%%%%%%%%%%%%
\section{Weakly reflective  submanifolds via the parallel transport map I} \label{wr}

In this section under suitable  assumptions we show that 
a submanifold of a compact normal homogeneous space is weakly reflective if and only if its inverse image under the parallel transport map $\Phi_K$ is a weakly reflective PF submanifold of $V_\mathfrak{g}$.

Let $G$, $V_\mathfrak{g}$, $\mathcal{G}$, $K$ be as in Section \ref{prel}. We consider the following  three actions.
\begin{enumerate}
\item[1.] $\mathcal{G}$ acts on $V_\mathfrak{g}$ by $g * u := g u g^{-1} - g' g^{-1}$ for $g \in \mathcal{G}$ and $u \in V_\mathfrak{g}$.
\item[2.] $G \times G$ acts on $G$ by $(b_1, b_2) \cdot a := b_1 a b_2^{-1}$ for $a, b_1, b_2 \in G$.
\item[3.] $G$ acts on $G/K$ by $b \cdot (aK) := (ba)K$ for $a, b \in G$.
\end{enumerate}
If a closed subgroup of $\mathcal{G}$, $G \times G$ or $G$ is given, then we consider the induced action. Let $\hat{G}:= \{\hat{b} \in \mathcal{G} \mid b \in G\}$ be the set of constant paths in $G$  and $\Delta G := \{(b,b) \mid b \in G\}$.
\begin{enumerate}
\item[1.] $\mathcal{G}_u = g \hat{G}g^{-1}$ denotes the isotropy subgroup of $\mathcal{G}$ at $u = g * \hat{0} \in V_\mathfrak{g}$, where $g \in \mathcal{G}$.
\item[2.] $(G \times G)_a =(a,e) \Delta G (a,e)^{-1}$ denotes the isotropy subgroup of $G \times G$ at $a \in G$.
\item[3.] $G_{aK} = a K a^{-1}$ denotes the isotropy subgroup of $G$ at $aK \in G/K$.
\end{enumerate}

\begin{thm} \label{type2}
Let $G$, $V_\mathfrak{g}$, $\mathcal{G}$, $\Phi$, $K$, $\Phi_K$ be as in Section $\ref{prel}$. 
\begin{enumerate}

\item[\textup{(i)}] Let $N$ be a closed submanifold of $G$. The following are equivalent.

\begin{enumerate}
\item[\textup{(a)}] $N$ is a weakly reflective submanifold of $G$ such that for each $(a, \xi) \in T^\perp N$,  a reflection $\nu_\xi$ of $N$ with respect to $\xi$ belongs to $(G \times G)_a$. 

\item[\textup{(b)}] $\Phi^{-1}(N)$ is a weakly reflective PF submanifold of $V_\mathfrak{g}$ such that for each $(u, X) \in T^\perp \Phi^{-1}(N)$, a reflection $\nu_X$ of $\Phi^{-1}(N)$ with respect to $X$ belongs to $\mathcal{G}_u$. 
\end{enumerate}

\item[\textup{(ii)}] Let $N$ be a closed submanifold of $G/K$. The following are equivalent:
\begin{enumerate}
\item[\textup{(a)}] $N$ is a weakly reflective submanifold of $G/K$ such that for each $(aK, w) \in T^\perp N$, a reflection $\nu_w$ of $N$ with respect to $w$ belongs to $G_{aK}$.

\item[\textup{(b)}] $\pi^{-1}(N)$ is a weakly reflective submanifold of $G$ such that for each $(a, \xi) \in T^\perp \pi^{-1}(N)$, a reflection $\nu_\xi$ of $N$ with respect to $\xi$ belongs to $(a, e) \Delta K(a, e)^{-1}$ $(\subset (a,e) \Delta G (a,e)^{-1} = (G \times G)_a)$.

\item[\textup{(c)}] $\Phi_K^{-1}(N)$ is a weakly reflective PF submanifold of $V_\mathfrak{g}$ such that for each $(u, X) \in T^\perp \Phi_K^{-1}(N)$ a reflection $\nu_X$ of $\Phi_K^{-1}(N)$ with respect to $X$ belongs to $g\hat{K}g^{-1}$ $(\subset g \hat{G} g^{-1} = \mathcal{G}_u)$, where $g \in \mathcal{G}$ satisfies $u = g * \hat{0}$.
\end{enumerate}
\end{enumerate}
\end{thm}
\begin{proof}
(i) (a) $\Rightarrow$ (b): Let $(u, X) \in T^\perp \Phi^{-1}(N)$. Choose $g \in P(G, G \times \{e\})$ so that $u = g * \hat{0}$. Set $a := \Phi(u) = g(0)$,  $N' := a^{-1}N$ and $\eta := a^{-1}(d \Phi(X)) \in T_e^\perp N'$. The horizontal lift of $\eta$ at $\hat{0} \in V_\mathfrak{g}$ is the constant path $\hat{\eta} \in T^\perp_{\hat{0}} \Phi^{-1}(N)$. By commutativity of (\ref{commute1}) we have $g * (\Phi^{-1}(N')) = \Phi^{-1}(N)$ and $(d g*) \hat{\eta} = X$. Thus in order to show the existence of a reflection $\nu_X$ of $\Phi^{-1}(N)$ with respect to $X$ as an element of $\mathcal{G}_u = g \hat{G}g^{-1}$, it suffices to construct a reflection $\nu_{\hat{\eta}}$ of $\Phi^{-1}(N')$ with respect to $\hat{\eta}$ as an element of $\mathcal{G}_{\hat{0}} = \hat{G}$. Let $\nu_{d \Phi(X)}$ be a reflection of $N$ with respect to $d \Phi(X)$ which is given by $\nu_\eta(c) = b' c b^{-1}$ for some $(b', b) \in (G \times G)_a$. Then a  reflection $\nu_{\eta}$ of $N'$ with respect to $\eta$ is defined  by 
$
\nu_{\eta} := (a, e)^{-1} \circ \nu_{d \Phi(X)} \circ (a,e)
$,
that is, $\nu_{\eta}(c) := b c b^{-1}$ for $c \in G$. Note that $\nu_\eta \in (G \times G)_e$. Define a linear orthogonal transformation $\nu_{\hat{\eta}}$ of $V_\mathfrak{g}$ by 
\begin{equation*}
\nu_{\hat{\eta}} (u) 
:= 
d \nu_\eta \circ u
=
b u b^{-1}
=
\hat{b} * u,
\quad 
u \in {V_\mathfrak{g}}.
\end{equation*}
Note that $\nu_{\hat{\eta}} \in \mathcal{G}_{\hat{0}}$.
Further by (\ref{equiv4}) (i) the following  diagram commutes.
\begin{equation*}
\begin{CD}
V_\mathfrak{g} @>\nu_{\hat{\eta}}>> V_\mathfrak{g}
\\
@V \Phi VV @V \Phi VV
\\
G @>\nu_\eta >> G
\end{CD}
\end{equation*}
Thus by Lemma \ref{weaksubm}, $\nu_{\hat{\eta}}$ is a reflection of $\Phi^{-1}(N')$ with respect to $\hat{\eta}$ and (b) follows.

(i) (b) $\Rightarrow$ (a): Let $(a, \xi) \in T^\perp N$. Set $N' := a^{-1}N$, $\eta := a^{-1}\xi \in T^\perp _e N'$. Fix $u \in \Phi^{-1}(a)$. Choose $g \in P(G, G \times \{e\})$ so that $u = g * \hat{0}$.
Let $X \in T^\perp_u \Phi^{-1}(N)$ be the horizontal lift of $\xi$ at $u$. Let $\nu_X$ be a reflection of $\Phi^{-1}(N)$ with respect to $X$ such that $\nu_X \in \mathcal{G}_u$. By commutativity of (\ref{commute1}) we have $g * \Phi^{-1}(N') = \Phi^{-1}(N)$ and $d(g *) \hat{\eta} = X$. Thus a reflection $\nu_{\hat{\eta}}$ of $\Phi^{-1}(N')$ with respect to $\hat{\eta}$ is defined by $\nu_{\hat{\eta}} := (g*)^{-1} \circ \nu_X \circ (g*)$. Since $\nu_{\hat{\eta}} \in \mathcal{G}_{\hat{0}}$ there exists $b \in G$ such that $\nu_{\hat{\eta}}(u) = b u b^{-1}$. Thus if we define an isometry $\nu_\xi$ of $G$ by $\nu_\eta(c) :=bcb^{-1}$ for $c \in G$, then it follows by Lemma \ref{weaksubm} that $\nu_\eta$ is a reflection of $N'$ with respect to $\eta$ and $\nu_\eta \in (G \times G)_e$. Therefore a reflection $\nu_\xi$ of $N$ with respect to $\xi$ is defined by $\nu_{\xi} := l_a \circ \nu_\eta \circ l_a^{-1}$ so that $\nu_\xi \in (G \times G)_a$. This proves (a).

(ii) (a) $\Rightarrow$ (b):  Let $(a, \xi) \in T^\perp \pi^{-1}(N)$. Let $\nu_{d \pi (\xi)}$ be a reflection of $N$ which is given by $\nu_{d \pi (\xi)}(cK) = (bc)K$ for some $ b \in G_{aK}$. Since $G_{aK} = aKa^{-1}$, there is $k \in K$ such that $b = aka^{-1}$. Define an isometry $\nu_\xi$ of $G$ by 
\begin{equation*}
\nu_\xi(c) := (aka^{-1}, k) \cdot c,
\quad c \in G.
\end{equation*}
Note that $\nu_\xi \in (a,e) \Delta K (a,e)^{-1}$. Moreover the following diagram commutes.
\begin{equation*}
\begin{CD}
G @>\nu_\xi>> G
\\
@V \pi VV @V \pi VV
\\
G/K @>\nu_{d \pi (\xi)}>> G/K
\end{CD}
\end{equation*}
Thus by Lemma \ref{weaksubm}, $\nu_\xi$ is a reflection of $\pi^{-1}(N)$ with respect to $\xi$. This proves (b).

(ii) (b) $\Rightarrow$ (a): Let $(aK, w) \in T^\perp N$. Let $\xi \in T^\perp_a \pi^{-1}(N)$ be the horizontal lift of $w$. Choose a reflection $\nu_\xi$ of $\pi^{-1}(N)$ with respect to $\xi$ which is given by $\nu_\xi (c) = (aka^{-1}, k) \cdot c$  for some $(aka^{-1}, k) \in (a,e) \Delta K (a,e)^{-1}$. Define an isometry $\nu_w$ of $G/K$ by $\nu_w(cK) :=aka^{-1} cK$. Then by Lemma \ref{weaksubm}, $\nu_w$ is a reflection of $N$ with respect to $w$. Since $aka^{-1} \in G_{aK}$, (a) follows.

The equivalence of (b) and (c) of (ii) follows by the similar arguments to (i). 
\end{proof}

For our  purpose of obtaining weakly reflective PF submanifolds, we give a corollary of Theorem \ref{type2} as follows. 
\begin{cor}\label{ctype2}
Let $G$, $H$, $K$ be as in Section $\ref{prel}$ and $K'$ a closed subgroup of $G$.
\begin{enumerate}
\item[\textup{(i)}] Suppose that an orbit $H \cdot a$ through $a \in G$ is a weakly reflective submanifold of $G$ such that for each $\xi \in T_a^\perp (H \cdot a)$,  a reflection $\nu_\xi$ of $H \cdot a$ with respect to $\xi$ belongs to $(G \times G)_a$. Then the orbit  $P(G,H) * u$ through $u \in \Phi^{-1}(a)$ is a weakly reflective PF submanifold of $V_\mathfrak{g}$ satisfying the condition in Theorem \textup{\ref{type2} (i) (b)}.

\item[\textup{(ii)}] Suppose that an orbit $K' \cdot aK$  through $aK \in G/K$ is a weakly reflective submanifold of $G/K$ such that for each $\xi \in T_a^\perp (K' \cdot aK)$, a reflection $\nu_\xi$ of $K' \cdot aK$ with respect to $\xi$ belongs to $G_{aK}$. Then the orbit $(K' \times K) \cdot a$ is a weakly reflective submanifold of $G$ satisfying the condition in Theorem \textup{\ref{type2} (ii) (b)}. Moreover the orbit $P(G, K' \times K)*u$ through $u \in \Phi^{-1}(a)$ is a weakly reflective PF submanifold of $V_\mathfrak{g}$ satisfying the condition in Theorem \textup{\ref{type2} (ii) (c)}.
\end{enumerate}
\end{cor}

Compared to Corollary \ref{ctype2}, 
the following theorem covers a somewhat different kind of weakly reflective orbits.   
\begin{thm}\label{type2.2}
Let $G$, $H$ be as in Section $\ref{prel}$. Suppose that an orbit $H \cdot e$ through $e \in G$ is a weakly reflective submanifold of $G$ such that for each $\xi \in T^\perp_e (H \cdot e)$, a reflection $\nu_\xi$ of $H \cdot e$ with respect to $\xi$ is an automorphism of $G$. Then the orbit  $P(G,H) * \hat{0}$ through $\hat{0} \in V_\mathfrak{g}$ is a weakly reflective PF submanifold of $V_\mathfrak{g}$. 
\end{thm}
\begin{proof}
Let $\nu_\xi$ be a reflection of $H \cdot e$ with respect to $\xi \in T^\perp_e (H \cdot e)$ which is an automorphism of $G$. Define a linear orthogonal transformation $\nu_{\hat{\xi}}$ of $V_\mathfrak{g}$ by 
\begin{equation} \label{reflection2}
\nu_{\hat{\xi}} (u) := d \nu_\xi \circ u, 
\quad 
u \in {V_\mathfrak{g}}.
\end{equation}
Since $\nu_\xi$ is an automorphism of $G$, we have  $
\nu_{\hat{\xi}}(g * \hat{0}) = (\nu_\xi \circ g) * \hat{0}$ for all $g \in \mathcal{G}$. This shows that the following diagram commutes.
\begin{equation*}
\begin{CD}
V_\mathfrak{g} @>\nu_{\hat{\xi}}>> V_\mathfrak{g}
\\
@V \Phi VV @V \Phi VV
\\
G @>\nu_\xi >> G
\end{CD}
\end{equation*}
Since $\nu_{\hat{\xi}}$ fixes $\hat{0} \in V_\mathfrak{g}$, it follows by Lemma \ref{weaksubm} that  $\nu_{\hat{\xi}}$ is a reflection of $P(G, H) * \hat{0}$ with respect to $\hat{\xi}$. By homogeneity of $P(G, H) * \hat{0}$,   our claim follows.
\end{proof}

In the rest of this section,  we see examples of Corollary \ref{ctype2} and Theorem \ref{type2.2}. 
\begin{example}\label{examp1}
It was proved (\cite[p.\ 442]{IST09}, \cite{Pod97}) that any singular orbit of a cohomogeneity one action is  weakly reflective. In this case each reflection is given by the action of the   isotropy subgroup. Thus by Corollary \ref{ctype2} we have the following examples.
\begin{enumerate}
\item[\textup{(i)}] Let $G$, $H$ be as in Section $\ref{prel}$. Suppose that the $H$-action is of cohomogeneity $1$. If an orbit $H \cdot a$ through $a \in G$  is  singular,  then $H \cdot a$ is a weakly reflective submanifold of $G$, and the orbit $P(G, H) * u$ through  $u \in \Phi^{-1}(a)$  is a weakly reflective PF submanifold of $V_\mathfrak{g}$.

\item[\textup{(ii)}] Let $G$, $K$, $K'$ be as in Corollary \ref{ctype2}.  Suppose that the $K'$-action is of  cohomogeneity $1$. If an orbit $K' \cdot aK$ through $a K \in G/K$ is singular,  then orbits $K' \cdot aK$ and $(K' \times K) \cdot a$ are weakly reflective submanifolds of $G/K$ and $G$, respectively. Moreover the orbit $P(G, K' \times K) * u$ through $u \in \Phi^{-1}(a)$ is a  weakly reflective PF submanifold of $V_\mathfrak{g}$.
\end{enumerate}
\end{example}

\begin{example}\label{examp2}
Let $G$ be a connected compact semisimple Lie group. Let $K = K_1$ and $K' = K_2$ be connected symmetric subgroups of $G$ with involutions $\theta_1$ and $\theta_2$, respectively. Suppose that $\theta_1 \circ \theta_2 = \theta_2 \circ \theta_1$. Ohno (\cite[Theorem 5]{Ohno16}) gave a sufficient condition for orbits  $(K_2 \times K_1) \cdot a$ and $K_2 \cdot aK_1$ to be  weakly reflective submanifolds of $G$ and $G/K_1$, respectively. By Corollary \ref{ctype2}, in this case  the orbits $P(G, K_2 \times K_1) * u$ through $u \in \Phi^{-1}(a)$ is a weakly reflective PF submanifold of $V_\mathfrak{g}$. 
\end{example}

\begin{example} \label{Ohno2}
Let $G$, $K_1$, $K_2$ be as in Example $\ref{examp2}$. Ohno  (\cite[Theorem 4]{Ohno16}) also gave another sufficient condition for an orbit $N :=(K_2 \times K_1) \cdot a$ to be a weakly reflective submanifold of $G$. In this case $\nu_a := l_a \circ \theta_1 \circ l_a^{-1}$ was shown to be a reflection of $N$ with respect to any normal vector at $a \in G$. 
Applying Theorem  $\ref{type2.2}$  to his result we can see that $\Phi^{-1}(N) =P(G, K_2 \times K_1) * u $ ($u \in \Phi^{-1}(a)$) is a weakly reflective PF submanifold of $V_\mathfrak{g}$ as follows. Choose $g \in P(G, G \times \{e\})$ so that $u = g * \hat{0}$. Then $a = \Phi(u) = g(0)$. 
Set $N' := a^{-1}N = ((a^{-1}K_2 a) \times K_1) \cdot e$. Then $\theta_1$ is a reflection of $N'$ with respect to any normal vector at $e \in N'$.  Since $\theta_1$ is an automorphism of $G$, it follows from Theorem \ref{type2.2} that $\Phi^{-1}(N')$ is a weakly reflective PF submanifold of $V_\mathfrak{g}$. By commutativity of \eqref{commute1} we have $g* \Phi^{-1}(N') =\Phi^{-1}(N)$. Thus $\Phi^{-1}(N)$ is a weakly reflective PF submanifold of $V_\mathfrak{g}$.
\end{example}

\section{Weakly reflective submanifolds via the parallel transport map II}

In this section supposing that $G/K$ is a Riemannian symmetric space of compact type we show that for \emph{any} weakly reflective submanifold $N$ of $G/K$ its inverse image under the parallel transport map $\Phi_K$ is a weakly reflective PF submanifold of $V_\mathfrak{g}$.

\begin{thm}\label{type3}
Let $G$, $V_\mathfrak{g}$, $K$, $\pi$, $\Phi_K$ be as in Section $\ref{prel}$. Suppose that $G$ is semisimple and  its bi-invariant Riemannian metric is induced by the negative multiple of the Killing form of $\mathfrak{g}$. Assume that $(G,K)$ is an effective symmetric pair. If $N$ is a weakly reflective submanifold of $G/K$, then
\begin{enumerate}
\item[\textup{(i)}] $\pi^{-1}(N)$ is a weakly reflective submanifold of $G$,

\item[\textup{(ii)}] $\Phi_K^{-1}(N)$ is a weakly reflective PF submanifold of $V_\mathfrak{g}$.
\end{enumerate}
\end{thm}
\begin{cor} Let $M$ be an irreducible Riemannian symmetric space of compact type \textup{(cf.\ \cite{Hel01})}. Denote by  $G$ the identity component of the group of isometries of $M$. Set $K := \{a \in G \mid L_a (p) =p \}$ for a fixed $p \in M$. Let $\Phi_K : V_\mathfrak{g}\rightarrow G/K = M$ be the parallel transport map.
If $N$ is a weakly reflective submanifold of $M$, then $\Phi_K^{-1}(N)$ is a weakly reflective PF submanifold of $V_\mathfrak{g}$.
\end{cor}

\begin{proof}[Proof of Theorem $\ref{type3}$]
(i) Let $(a, \xi) \in T^\perp \pi^{-1}(N)$. Denote by $l_a$ the left translation by $a \in G$ and $L_a$ an isometry on $G/K$ defined by $L_a(bK) := abK$ for $b \in G$. Set $N' := L_a^{-1}(N)$. Let $\eta \in T_e^\perp \pi^{-1}(N')$ be the horizontal lift of $d L_a^{-1} \circ d \pi (\xi) \in T^\perp _{eK}N'$. By commutativity of (\ref{commute2})  we have $l_a (\pi^{-1}(N')) = \pi^{-1}(N)$ and $d l_a (\eta) = \xi$. Thus in order to show the existence of a reflection $\nu_\xi$ of $\pi^{-1}(N)$ with respect to $\xi$, it suffices to construct a reflection $\nu_\eta$ of $\pi^{-1}(N')$ with respect to $\eta$. Let $\nu_{d \pi (\xi)}$ be a reflection of $N$ with respect to $d \pi (\xi) \in T^\perp_{aK} N$. Define  a reflection $\nu_{d \pi (\eta)}$ of $N'$ with respect to $d \pi (\eta) = d L_a^{-1} \circ d \pi(\xi) \in T^\perp_{eK} N'$ by
$
\nu_{d \pi (\eta)} := L_a^{-1} \circ \nu_{d \pi (\xi)} \circ L_a.
$
Now we define $\nu_\eta$ as follows. Denote by $I(G/K)$  the group of isometries of $G/K$ with identity component $I_0(G/K)$. 
By the assumption, a map $L : G \rightarrow I(G/K)$, $a \mapsto L_a$ is a Lie group isomorphism onto $I_0(G/K)$ (\cite[p.\ 243]{Hel01}). Since $I_0(G/K)$ is a normal subgroup of $I(G/K)$, we can define a map $\nu_\eta: G \rightarrow G$, $b \mapsto \nu_\eta(b)$ by
\begin{equation}\label{reflection3}
L_{\nu_\eta(b)} := \nu_{d \pi (\eta)} \circ L_b \circ \nu_{d \pi (\eta)}^{-1}.
\end{equation}
Note that $\nu_\eta$ is an automorphism of $G$ and thus an isometry of $G$ which fixes $e \in G$. Moreover since $L(K) = \{f \in I_0(G/K) \mid f(eK) =eK\}$ and $\nu_{d\pi(\eta)}$ fixes $eK$, we have  $\nu_\eta(K) \subset K$. Furthermore it follows that the induced map on $G/K$ from $\nu_{\eta}$ is identical to $\nu_{d\pi(\eta)}$. Thus by Lemma \ref{weaksubm}, $\nu_\eta$ is a reflection of $\pi^{-1}(N')$ with respect to $\eta$. This proves (i). 

(ii) Let $(u, X) \in T^\perp \Phi_K^{-1}(N)$. Choose $g \in P(G, G \times \{e\})$ so that $u = g * \hat{0}$. Set $a := \Phi(u) = g(0)$ and $N' := L_a^{-1}(N)$. Let $\eta \in T^\perp_{e} \pi^{-1}(N')$ be the horizontal lift of $d L_a^{-1} \circ d \Phi_K(X) \in T^\perp_{eK} N'$ with respect to the Riemannian submersion $\pi :G \rightarrow G/K$.  Further with respect to the Riemannian submersion $\Phi:V_\mathfrak{g} \rightarrow G$ the horizontal lift of $\eta$ at $\hat{0} \in V_\mathfrak{g}$ is the constant path $\hat{\eta} \in T^\perp_{\hat{0}} \Phi_{K}^{-1}(N')$. By commutativity of (\ref{commute3}) we have $ g*\Phi_K^{-1}(N') = \Phi^{-1}_{K}(N)$ and $d(g *)\hat{\eta} = X$. Thus in order to show the existence of a reflection $\nu_X$ of $\Phi_{K}^{-1}(N)$ with respect to $X$, it suffices to construct a reflection $\nu_{\hat{\eta}}$ of $\Phi_{K}^{-1}(N')$ with respect to $\hat{\eta}$. 
By the same way as in (i) we can define a reflection $\nu_\eta$ of $\pi^{-1}(N')$ with respect to $\eta$ . Since $\nu_\eta$ is an automorphism of $G$, we can also define a reflection $\nu_{\hat{\eta}}$ of $\Phi_{K}^{-1}(N')$ with respect to $\hat{\eta}$ similarly to (\ref{reflection2}). This proves (ii).
\end{proof}

\begin{rem}
Even if $N$ is reflective in Theorem \ref{type3}, $\Phi^{-1}_K(N)$ can not be reflective due to Corollary \ref{tgeod7} (ii).  In this case there exists one more reflective submanifold $N^\perp$  of $G/K$ corresponding to $N$ (\cite[p.\ 328]{Leu74}) and thus a pair of two weakly reflective PF submanifolds appears in the Hilbert space $V_\mathfrak{g}$.
\end{rem}

\begin{rem}
Let $G$, $K$, $\Phi_K$ be as in Section \ref{prel}. The fact that each fiber of $\Phi_K$ is weakly reflective also follows from Theorem \ref{type3} if $(G,K)$ satisfies the assumptions in Theorem \ref{type3}. The advantage of Corollary \ref{weakfiber} is that it does not require such assumptions. It is also noted that under such assumptions each of the fibers has at least two different weakly reflective structures. 
\end{rem}

\begin{example}\label{examp3} 
Ikawa, Sakai and Tasaki (\cite[Theorem 4]{IST09}) classified weakly reflective submanifolds of the standard sphere given as orbits of $s$-representations of irreducible Riemannian symmetric pairs. Applying Theorem \ref{type3} to their result we obtain weakly reflective PF submanifolds as follows. Let $(U, L)$ be a compact Riemannian symmetric pair. Suppose that $L$ is connected. Denote by $\mathfrak{u} = \mathfrak{l} \oplus \mathfrak{p}$ the canonical decomposition and $\operatorname{Ad}: L \rightarrow SO(\mathfrak{p})$ the isotropy representation. If an orbit $\operatorname{Ad}(L) \cdot x$ through $x \in \mathfrak{p}$ is a weakly reflective submanifold of the hypersphere $S(\|x\|)$ in $\mathfrak{p}$, then the orbit $P(SO(\mathfrak{p}), \operatorname{Ad}(L) \times SO(\mathfrak{p})_x) * \hat{0}$ is a weakly reflective PF submanifold of the Hilbert space $V_{ \mathfrak{so}(\mathfrak{p})}$. 
\end{example}

\begin{example}\label{examp4}  
Enoyoshi (\cite[Proposition 4]{Eno18}) gave an example of a weakly reflective submanifold in a symmetric space $SO(7) / SO(3) \times SO(4)$ by the action of the exceptional Lie group $G_2$. Applying Theorem \ref{type3} to her result an orbit $P(SO(7), G_2 \times (SO(3) \times SO(4))) * \hat{0}$ is a weakly reflective PF submanifold of the Hilbert space $V_{\mathfrak{so}(7)}$.
\end{example}

\begin{rem}
In Theorems \ref{type2}, \ref{type2.2} and \ref{type3}, suppose further that $N$ is a weakly reflective submanifold such that at each point  there exists  a reflection which is independent of the choice of normal vectors. Then the corresponding weakly reflective PF submanifolds also have such a property.
\end{rem}

%%%%%%%%%%%%%%%%%%%%%%%%
\section*{Acknowledgements}
The author would like to thank
Professor Yoshihiro Ohnita for helpful discussions and many invaluable suggestions. The author is also grateful to Professors Takashi Sakai and Hiroshi Tamaru for their interests in my work and their useful comments. 
%%%%%%%%%%%%%%%%%%%%%%%%


\begin{thebibliography}{9}
\bibitem{BB85} B. Booss, D. D.  Bleecker, \emph{Topology and Analysis. The Atiyah-Singer index formula and gauge-theoretic physics}, Translated from the German by Bleecker and A. Mader. Universitext. Springer-Verlag, New York, 1985. 

\bibitem{Con64} L. Conlon, \emph{The topology of certain spaces of paths on a compact symmetric space}, Trans. Amer. Math. Soc. \textbf{112} (1964) 228-248. 

\bibitem{Eno18} K. Enoyoshi, \emph{Principal curvatures of homogeneous hypersurfaces in a
Grassmann manifold $\widetilde {\operatorname{Gr}}_3 (\operatorname{Im}\mathbb{O})$ by the $G_2$-action}, to appear in Tokyo J. Math.

\bibitem{HL82} R. Harvey and H. B. Lawson, Jr., \emph{Calibrated geometries}, Acta Math., \textbf{148} (1982), 47-157.

\bibitem{HLO06} E. Heintze, X. Liu, C. Olmos,  \emph{Isoparametric submanifolds and a Chevalley-type restriction theorem}, Integrable systems, geometry, and topology, 151-190, AMS/IP Stud. Adv. Math., \textbf{36}, Amer. Math. Soc., Providence, RI, 2006.

\bibitem{Hel01} S. Helgason, \emph{Differential Geometry, Lie groups, and Symmetric Spaces}, Corrected reprint of the 1978 original. Graduate Studies in Mathematics, \textbf{34}. American Mathematical Society, Providence, RI, 2001.

\bibitem{IST09}  O. Ikawa,  T. Sakai, H. Tasaki, \emph{Weakly reflective submanifolds and austere submanifolds}. J. Math. Soc. Japan \textbf{61} (2009), no. 2, 437-481.

\bibitem{KN96II} S. Kobayashi, K. Nomizu, \emph{Foundations of Differential Geometry}. Vol. I and Vol. II. Reprint of the 1969 original. Wiley Classics Library. A Wiley-Interscience Publication. John Wiley \& Sons, Inc., New York, 1996. 

\bibitem{Koi02} N. Koike, \emph{On proper Fredholm submanifolds in a Hilbert space arising from submanifolds in a symmetric space}, Japan. J. Math. (N.S.) \textbf{28} (2002), no. 1, 61-80.

\bibitem{KT93} C. King, C.-L. Terng, \emph{Minimal submanifolds in path space}, Global analysis in modern mathematics, 253-281, Publish or Perish, Houston, TX, 1993.

\bibitem{Leu73} Dominic S. P. Leung,  \emph{The reflection principle for minimal submanifolds of Riemannian symmetric spaces},  J. Differential Geom. \textbf{8} (1973), 153-160. 

\bibitem{Leu74} Dominic S. P. Leung, \emph{On the classification of reflective submanifolds of Riemannian symmetric spaces},  Indiana Univ. Math. J. \textbf{24} (1974/75), 327-339.

\bibitem{Ohno16} S. Ohno,  \emph{A sufficient condition for orbits of Hermann actions to be weakly reflective}, Tokyo J. Math. \textbf{39} (2016), no. 2, 537-564.

\bibitem{PT88} R. Palais, C.-L. Terng, \emph{Critical Point Theory and Submanifold Geometry}, Lecture Notes in Math., vol 1353, Springer-Verlag, Berlin and New York, 1988.

\bibitem{PiTh90} U. Pinkall, G. Thorbergsson, \emph{Examples of infinite dimensional isoparametric submanifolds},  Math. Z. \textbf{205} (1990), no. 2, 279-286. 

\bibitem{Pod97} F. Podest\`a, \emph{Some remarks on austere submanifolds}, Boll. Un. Mat. Ital. B (7) \textbf{11} (1997), no. 2, suppl., 157-160.

\bibitem{Sca02} A. J. D. Scala, \emph{Minimal homogeneous submanifolds in Euclidean spaces}, Ann. Global Anal. Geom. \textbf{21} (2002), no. 1, 15-18.

\bibitem{Ter89} C.-L.  Terng,  \emph{Proper Fredholm submanifolds of Hilbert space.} J. Differential Geom. \textbf{29} (1989), no. 1, 9-47.

\bibitem{Ter95} C.-L. Terng, \emph{Polar actions on Hilbert space}. J. Geom. Anal. \textbf{5} (1995), no. 1, 129-150. 

\bibitem{TT95} C.-L. Terng, G. Thorbergsson, \emph{Submanifold geometry in symmetric spaces.} J. Differential Geom. \textbf{42} (1995), no. 3, 665-718.
\end{thebibliography}
\end{document}